\documentclass{amsart}

\usepackage{color}

\usepackage{amsxtra,amscd,amsthm}
\usepackage{mathtools}

\usepackage{leftidx}
\mathtoolsset{showonlyrefs} 

\usepackage{mathtools}
\usepackage{nicefrac}
\usepackage{amssymb}
\usepackage{color}

\newtheorem{teo}{Theorem}[section]
\newtheorem{lem}[teo]{Lemma}
\newtheorem{co}[teo]{Corollary}

\theoremstyle{remark}
\newtheorem{re}[teo]{Remark}

\theoremstyle{definition}

\newcommand{\R}{\mathbb{R}}
\newcommand{\wrsp}{\mathcal{W}^{(r,s)}_p}
\newcommand{\wrsi}{\mathcal{W}^{(r,s)}_\infty}

\def\dist{\mathop{\mbox{\normalfont dist}}\nolimits}

\title[{Eigenvalues for systems of fractional $p-$Laplacians}]{Eigenvalues for systems of fractional $p-$Laplacians}
\author[Leandro Del Pezzo and Julio D. Rossi]{Leandro Del Pezzo and Julio D. Rossi}

\address{Leandro M. Del Pezzo \and Julio D. Rossi
\hfill\break\indent
CONICET and Departamento  de Matem{\'a}tica, FCEyN,
Universidad de Buenos Aires,
\hfill\break\indent Pabellon I, Ciudad Universitaria (1428),
Buenos Aires, Argentina.}

\email{{\tt ldpezzo@dm.uba.ar,
jrossi@dm.uba.ar
}}

\thanks{
 Julio D. Rossi  was partially supported by MTM2011-27998,
(Spain) }

\begin{document}

\begin{abstract}
We study the eigenvalue problem for a system of fractional $p-$Laplacians, that is,
$$
\begin{cases}
			(-\Delta_p)^r u = \lambda\dfrac{\alpha}p|u|^{\alpha-2}u|v|^{\beta} &\text{in } \Omega,\vspace{.1cm}\\ 
			(-\Delta_p)^s u = \lambda\dfrac{\beta}p|u|^{\alpha}|v|^{\beta-2}v &\text{in } \Omega,\\
			u=v=0 &\text{in }\Omega^c=\R^N\setminus\Omega.	
		\end{cases}
		$$
		We show that there is a first (smallest) eigenvalue that is simple and has associated eigen-pairs composed  of positive and bounded 
		functions. Moreover, there is a sequence of eigenvalues $\lambda_n$ such that
		$\lambda_n\to\infty$ as $n\to\infty$.
		
		In addition, we study the limit as $p\to \infty$ of the first eigenvalue, $\lambda_{1,p}$, and we obtain 
		$
			[\lambda_{1,p}]^{\nicefrac{1}{p}}\to \Lambda_{1,\infty}
		$
		as $p\to\infty,$ where
		$$
			\Lambda_{1,\infty} 
			=  \inf_{(u,v)} \left\{
			\frac{\max \{ [u]_{r,\infty} ; [v]_{s,\infty} \} }{ \| |u|^{\Gamma} |v|^{1-\Gamma} \|_{L^\infty (\Omega)} }
			\right\} =  \left[ \frac{1}{R(\Omega)} \right]^{ (1-\Gamma) s + \Gamma r }.
		$$
		Here $R(\Omega):=\max_{x\in\Omega}\dist(x,\partial\Omega)$ and $[w]_{t,\infty} \coloneqq \sup_{(x,y)\in\overline{\Omega}} \frac{| w(y) - w(x)|}{|x-y|^{t}}.$
		
		Finally, we identify a PDE problem satisfied, in the viscosity sense, by any possible uniform limit 
	along subsequences of the eigen-pairs.
\end{abstract}

\maketitle


\section{Introduction}

	In this work we deal the non-local non-linear eigenvalue problem
	\begin{equation}
		\label{eq:autovalores}
		\begin{cases}
			(-\Delta_p)^r u = \lambda\dfrac{\alpha}p|u|^{\alpha-2}u|v|^{\beta} &\text{in } \Omega,\vspace{.2cm}\\ 
			(-\Delta_p)^s u = \lambda\dfrac{\beta}p|u|^{\alpha}|v|^{\beta-2}v &\text{in } \Omega,\\
			u=v=0 &\text{in }\Omega^c=\R^N\setminus\Omega,	
		\end{cases}
	\end{equation} 
	where $p>1,$ $r,s\in(0,1),$ $\alpha,\beta\in(0,p)$ are such that
	\begin{equation} \label{eq:a.b}
		\alpha + \beta = p, \qquad \min\{\alpha;\beta\}\ge1, 
	\end{equation}
	and $\lambda$ is the eigenvalue. 
	Here and subsequently $\Omega$ is a bounded smooth domain in $\R^N$ and 
	$(-\Delta_p)^t$ denotes the fractional
	$(p,t)-$Laplacian, that is
	\[
		(-\Delta_p)^t u(x)\coloneqq 2\text{P.V.} \int_{\R^N}
		\dfrac{|u(x)-u(y)|^{p-2}(u(x)-u(y))}{|x-y|^{N+sp}} \, dy
		\quad x\in\Omega.
	\]

	The natural functional space for our problem is  
	\[
		\wrsp(\Omega)\coloneqq\widetilde{W}^{r,p}(\Omega)\times\widetilde{W}^{s,p}(\Omega).
	\]
	Here $\widetilde{W}^{t,p}(\Omega)$ denotes the space of all 
	$u$ belong to the fractional Sobolev space
	\[
		W^{t,p}(\Omega)\coloneqq
		\left\{
			v\in L^p(\Omega)\colon
			\int_{\Omega^2}\dfrac{|v(x)-v(y)|^p}{|x-y|^{N+tp}} dxdy<\infty
		\right\} 
	\]
	such that $\tilde{u}\in W^{t,p}(\R^N)$ where $\tilde{u}$ is the extension by 
	zero of $u$ and $\Omega^2=\Omega\times\Omega.$ 
	For a more detailed description of  these spaces and some its properties, see 
	for instance \cite{Adams,Hitchhiker}.
	
	Note that in our eigenvalue problem we are considering two different fractional operators (since we
	allow for $t\neq s$) and therefore the natural space to consider here, that is 
	$\wrsp(\Omega)=\widetilde{W}^{r,p}(\Omega)\times\widetilde{W}^{s,p}(\Omega)$, 
	is not symmetric.

	In this context, an eigenvalue is a real value $\lambda$ for which 
	there is $(u,v)\in\wrsp(\Omega)$ such that $uv\not\equiv0,$  and
	$(u,v)$ is a weak solution of \eqref{eq:autovalores}, i.e.,
	\begin{align*}
		\int_{\R^{2N}}
		\dfrac{|u(x)-u(y)|^{p-2}(u(x)-u(y))(w(x)-w(y))}{|x-y|^{N+rp}}dxdy
		&=\lambda\dfrac\alpha{p}\int_{\Omega}|u|^{\alpha-2}u|v|^{\beta}w dx\\
		\int_{\R^{2N}}
		\dfrac{|v(x)-v(y)|^{p-2}(v(x)-v(y))(z(x)-z(y))}{|x-y|^{N+sp}}dxdy
				&=\lambda\dfrac{\beta}p\int_{\Omega}|u|^{\alpha}|v|^{\beta-2}vz dx
	\end{align*} 
	for any $(w,z)\in\wrsp(\Omega).$ The pair $(u,v)$ is called a corresponding eigenpair.
	
	Observe that if $\lambda$ is an eigenvalue with eigenpair $(u,v)$
	then $uv\not\equiv0$ and 
	\[
		\lambda=\dfrac{[u]_{r,p}^p+[v]_{s,p}^p}
		{|(u,v)|_{\alpha,\beta}^p},
	\]	
	here 
	\[
		[w]_{t,p}^p\coloneqq\int_{\mathbb{R}^{2N}}
		\dfrac{|w(x)-w(y)|^p}{|x-y|^{N+tp}} dxdy 
		\quad \text{ and }\quad	
		|(u,v)|_{\alpha,\beta}^p\coloneqq\int_{\Omega} |u|^{\alpha}|v|^{\beta}dx.
	\]
	Thus
	\[
		\lambda\ge \lambda_{1,p}
	\]
	where
	\begin{equation}\label{eq:1eraut}
			\lambda_{1,p}\coloneqq
				\inf\left\{
					\dfrac{[u]_{r,p}^p+[v]_{s,p}^p}
							{|(u,v)|_{\alpha,\beta}^p}\colon
							(u,v)\in\wrsp(\Omega), uv\not\equiv0
				\right\}.
	\end{equation}
	
	Our first aim is to show that $\lambda_{1,p}$ is the first eigenvalue
	of our problem. In fact, in Section \ref{section:1erAutov}, we prove the following result.
	
	\begin{teo} \label{teo:1eraut}
		There is a nontrivial minimizer 
		$(u_p,v_p)$ of \eqref{eq:1eraut} such that both components are positives, $u_p,v_p > 0$ in $\Omega$,
		and $(u_p,v_p)$ is a weak solution of \eqref{eq:autovalores} with
		$\lambda=\lambda_{1,p}.$ 
		Moreover, $\lambda_{1,p}$ is simple.  
		
		Finally, there is a sequence of eigenvalues $\lambda_n$ such that
		$\lambda_n\to\infty$ as $n\to\infty$.
	\end{teo}
	We don't know if the first eigenvalue is isolated or not.
	
	\medskip
	
	Now, our aim is to study $\lambda_{1,p}$ for large $p$. To this end we look for the asymptotic behaviour of $\lambda_{1,p}$ as $p\to \infty$. From now on for any $p>1,$  $(u_p,v_p)$ denotes the eigen-pair associated to $\lambda_{1,p}$
	such that $|(u,v)|_{\alpha,\beta}=1.$ To study the limit as $p\to \infty$ we need to assume that
	\begin{equation}
		\label{eq:alfabeta}
		p\min\{r,s\}\ge N,
	\end{equation}		
	and
	\begin{equation}\label{lim.Gamma}
		\lim_{p\to \infty} \frac{\alpha_p}{p} = \Gamma, \qquad 0<\Gamma <1.
	\end{equation}
	Note that the last assumption and the fact that $\alpha_p +\beta_p=p$ implies 
	\begin{equation}\label{lim.Gamma.2}
		\lim_{p\to \infty} \frac{\beta_p}{p} = 1 - \Gamma, \qquad 0<1-\Gamma <1.
	\end{equation}

	In order to state our main theorem concerning the limit as $p\to \infty$, we need to introduce the following notations:
	\[
		[w]_{t,\infty} \coloneqq \sup_{(x,y)\in\overline{\Omega}} \frac{| w(y) - w(x)|}{|x-y|^{t}},
	\]
	\[
		\widetilde{W}^{t,\infty}(\Omega)\coloneqq 
		\left\{w\in C_0(\overline{\Omega})\colon 
		[w]_{t,\infty}<\infty,\right\},
		\quad \wrsi(\Omega)\coloneqq\widetilde{W}^{r,\infty}(\Omega)\times\widetilde{W}^{s,\infty}(\Omega)
	\]	
	and
	\[
		R(\Omega)\coloneqq\max_{x\in\Omega}\dist(x,\partial\Omega).
	\]
	
	Now we are ready to state our second result. It says that there is a limit for $[\lambda_{1,p}]^{\nicefrac{1}{p}}$ and that this limit verifies both a variational characterization and a simple geometrical characterization. In addition, concerning eigenfunctions there is a uniform limit (along subsequences) that is a viscosity solution to a limit PDE eigenvalue problem. The proofs of our results concerning limits as $p\to \infty$ are gathered in Section \ref{sec-p-infty}.
	
		\begin{teo} \label{teo.2.intro}
		Under the assumptions \eqref{eq:alfabeta} and \eqref{lim.Gamma}, we have that 	
		$$
			\lim_{p\to\infty } [\lambda_{1,p}]^{\nicefrac{1}{p}}= \Lambda_{1,\infty}
		$$
		where
		$$
			\Lambda_{1,\infty} 
			=  \inf \left\{
			\frac{\max \{ [u]_{r,\infty} ; [v]_{s,\infty} \} }{ \| |u|^{\Gamma} |v|^{1-\Gamma} \|_{L^\infty (\Omega)} }
			\colon (u,v)\in \wrsi(\Omega)\right\}.
		$$
		Moreover,  we have the following geometric characterization of the limit eigenvalue:
		$$
			\Lambda_{1,\infty}  = \left[ \frac{1}{R(\Omega)} \right]^{ (1-\Gamma) s + \Gamma r }.
		$$ 
		
		Lastly, there is a sequence $p_j \to \infty$ such that $(u_{p_j},v_{p_j})\to (u,v)$  
		converges uniformly in $\overline{\Omega},$ where 
		$(u_\infty,v_\infty)$ is a minimizer of $\Lambda_{1,\infty}$, and 
		a viscosity solution to
		\begin{equation}\label{eq:limite.intro}
			\begin{cases}
				\min\left\{\mathcal{L}_{r,\infty}u(x);\mathcal{L}_{r,\infty}^+u(x)-\Lambda_{1,\infty} u^{\Gamma}(x) v^{1-\Gamma}(x)\right\}
				=0
				&\text{ in } \Omega,\\
				\min\left\{\mathcal{L}_{s,\infty}u(x);\mathcal{L}_{s,\infty}^+u(x)-\Lambda_{1,\infty} u^{\Gamma}(x) v^{1-\Gamma}(x)\right\}
				=0&\text{ in } \Omega,\\
				u=v=0 &\text{ in } \R^N\setminus\Omega,
			\end{cases}
		\end{equation}
		where 
		\[
			\mathcal{L}_{t,\infty}w(x)\coloneqq\mathcal{L}_{t,\infty}^+w(x)
			+\mathcal{L}_{r,\infty}^-w(x)= \sup_{y\in\R^N}\dfrac{w(x)-w(y)}{|x-y|^{t}}
			+\inf_{y\in\R^N}\dfrac{w(x)-w(y)}{|x-y|^{t}}.
		\]
	\end{teo}

To end the introduction let us briefly refer to previous references on this subject.
	The limit of $p-$harmonic functions (solutions to the local $p-$Laplacian, that is,
	$-\Delta_p u =-\mbox{div} (|\nabla u|^{p-2} \nabla u)= 0$) as $p\to\infty$ has been extensively studied in the literature 
	(see \cite{BBM} and the survey \cite{ACJ}) and leads naturally to solutions of the infinity Laplacian, given by
	$-\Delta_{\infty} u = - \nabla u D^2 u (\nabla u)^t=0$. Infinity
	harmonic functions (solutions to $-\Delta_\infty u =0$) are
	related to the optimal Lipschitz extension problem (see the survey
	\cite{ACJ}) and find applications in optimal transportation, image
	processing and tug-of-war games (see, e.g.,\cite{CMS,GAMPR,PSSW,PSSW2} and the references therein).
	Also limits of the eigenvalue problem related to the $p$-Laplacian witth various boundary conditions 
	have been exhaustively examined, see \cite{GMPR,JL,JLM,RoSain,RoSain2}, 
	and lead naturally to the infinity Laplacian eigenvalue problem (in the scalar case)
	\begin{equation}\label{infty.1}
		\min \left\{ |\nabla u|  - \lambda u    ,\ - \Delta_{\infty} u 
		\right\}=0.
	\end{equation}
	In particular, the limit as $p\to \infty$ of the first eigenvalue $\lambda_{p,D}$ 
	of the $p$-Laplacian with Dirichlet boundary conditions and of its corresponding positive 
	normalized eigenfunction $u_p$ have been studied in \cite{JL,JLM}.  
	It was proved  there that, up to a subsequence,  the eigenfunctions $u_{p}$ converge uniformly to some Lipschitz  
	function $u_\infty$  satisfying $\|u_\infty\|_\infty=1$, and 
	\begin{equation}\label{DefInfEig} 
		(\lambda_{p,D})^{\nicefrac{1}{p}}  \to \lambda_{\infty,D}  
		= \inf_{u\in W^{1,\infty}(\Omega)} \dfrac{\|\nabla u\|_\infty}{\|u\|_\infty}
		= \dfrac{1}{R(\Omega)}. 
	\end{equation}   
	Moreover $u_\infty$ is an extremal for this limit variational problem and 
	the pair $u_\infty$, $\lambda_{\infty,D}$ is a
	nontrivial solution to \eqref{infty.1}. 
	This problem has also been studied from an optimal  mass-transport point of view in 
	\cite{ChdPG}.  Note that here the fact that we are dealing with two different operators in the system is reflected in that the limit is given by $$
			\Lambda_{1,\infty}  = \left[ \frac{1}{R(\Omega)} \right]^{ (1-\Gamma) s + \Gamma r },$$ a quantity that depends on $s$ and $t$.

	On the other hand, there is a rich recent literature concerning eigenvalues for systems of $p-$Laplacian type,
	(we refer e.g. to \cite{BdF,dpr,FMST,dNP,Z} and references therein). The only references that we know concerning the asymptotic behaviour as $p$ goes to infinity 
	of the eigenvalues for a system are \cite{BRS} and \cite{dpr} where the authors study the behaviour of the first eigenvalue for a system with the usual local $p-$Laplacian operator.
	
	Finally, concerning limits as $p\to \infty$ in fractional eigenvalue 	
	problems (a single equation), we mention \cite{Brasco,FP,JL}.
	In \cite{JL} the limit of the first eigenvalue for the fractional 
	$p-$Laplacian is studied while in \cite{FP} higher eigenvalues are 
	considered.

\section{Preliminaries}\label{A}

	We begin with a review of the basic results that 
	will be needed in subsequent sections.
	The known results are generally stated without proofs, 
	but we provide references where
	the proofs can be found. Also, 
	we introduce some of our notational conventions.

\subsection{Fractional Sobolev spaces}
	Let $s\in(0,1)$ and $p\in(1,\infty).$ 
	There are several choices for a norm for $W^{s,p}(\Omega),$ we choose the following:
	\[
		\|u\|_{s,p}^p\coloneqq \|u\|_{L^p(\Omega)}^p+|u|_{s,p}^p	
	\]
	where 
	\[
		|u|_{s,p}^p=\int_{\Omega^2}\dfrac{|u(x)-u(y)|^p}{|x-y|^p}\, dxdy. 
	\]
	
	Observe that for any $u\in\widetilde{W}^{s,p}(\Omega)$ we get
	\[
		|u|_{s,p}\le [u]_{s,p}.
	\]	
	
	Our first aim is to show a  Poincar\'e--type inequality.
	
	\begin{lem} \label{lem:poincare}
		Let $s\in(0,1).$ For any $p>1,$ there is a positive constant $C,$ independent of $p,$ such that
		\[
			[u]_{s,p}^p\ge \dfrac{\omega_N}{sp}({\rm{diam}}(\Omega)+1)^{sp}
			\|u\|_{L^p(\Omega)}^p
			\quad\forall u\in\widetilde{W}^{s,p}(\Omega)
		\]
		where $\omega_N$ is the $N-$dimensional volume of a Euclidean ball of radius 1.
	\end{lem}
	
	\begin{proof}
		Let $u\in\widetilde{W}^{s,p}(\Omega).$ Then
		\[
			[u]_{s,p}^p\geq \int_{\Omega}|u(x)|^p\int_{\Omega_1}\dfrac{1}{|x-y|^{N+sp}} dydx
		\]
		where $\Omega_1=\{y\in\Omega^c\colon \dist(y,\Omega)\ge1\}.$ Now, we observe that
		for any $x\in\Omega$ we have $B_{d+1}(x)^c\subset\Omega_1$ where $d={\rm{diam}}(\Omega).$ Thus
		\[
			\int_{\Omega_1}\dfrac{dy}{|x-y|^{N+sp}}\ge
			\int_{B_{d+1}(x)^c}\dfrac{dy}{|x-y|^{N+sp}}=
			\omega_{N}\int_{d+1}^{\infty} \dfrac{d\rho}{\rho^{sp+1}}=\dfrac{\omega_N}{sp}(d+1)^{sp}
		\]
		for all $x\in\Omega.$ Therefore, we conclude that,
		\[
			[u]_{s,p}^p\ge \dfrac{\omega_N}{sp}(d+1)^{sp}\|u\|_{L^p(\Omega)}^p.
		\]
	\end{proof}
	
	The following result will be one of the keys in the proof of Theorem \ref{teo.2.intro}.
	
	\begin{lem}\label{lem:inclusion}
		Let $s\in(0,1)$ and $p>\nicefrac{s}{N}.$ If $q\in(\nicefrac{N}{s},p)$ 
		and $t=s-\nicefrac{N}{q}$ then
		\[
			\|u\|_{L^{q}(\Omega)}\le |\Omega|^{\nicefrac1q-\nicefrac1p}
			\|u\|_{L^p(\Omega)} \qquad \text{ and } \qquad |u|_{t,q}\le
			{\rm{diam}}(\Omega)^{\nicefrac{N}p}|\Omega|^{\nicefrac2q-\nicefrac2p}|u|_{s,p}
		\]
		for all $u\in W^{s,p}(\Omega).$
	\end{lem}	
	\begin{proof}
		Since $q<p$, the first inequality is trivial, then, we only need to prove the second one.
		Let $u\in W^{s,p}(\Omega).$ It follows from H\"older's inequality that
		\begin{align*}
			|u|_{t,q}^q&=\int_{\Omega^2}\dfrac{|u(x)-u(y)|^q}{|x-y|^{sq}} dxdy\\
					&\le
					\left(\int_{\Omega^2}\dfrac{|u(x)-u(y)|^p}{|x-y|^{sp}} dxdy\right)^{\nicefrac{q}p}
					|\Omega|^{2-2\nicefrac{q}p}\\
					&\le\text{diam}(\Omega)^{\nicefrac{Nq}p}
					\left(\int_{\Omega^2}\dfrac{|u(x)-u(y)|^p}{|x-y|^{sp+N}} dxdy\right)^{\nicefrac{q}p}
					|\Omega|^{2-2\nicefrac{q}p},
		\end{align*}
		as we wanted to show.
	\end{proof}

\subsection{Weak and Viscosity Solutions} 

	Let us discuss the relation between the weak solutions of
	\begin{equation}\label{eq:viscosity1}
		\begin{cases}
			(-\Delta_p)^s u= f(x) &\text{ in }\Omega,\\
			u=0 &\text{ in }\Omega^c,
		\end{cases}
	\end{equation}
	and  the viscosity solutions of the same problem.
	
	\medskip
	
	We begin by introducing the precise definitions of weak and viscosity solutions.
	
	\medskip
	
	\noindent{\bf Definition (weak solution).}
	Let $f\in W^{-s,p}(\Omega)$ (the dual space of $\widetilde{W}^{s,p}(\Omega)$) and $u\in \widetilde{W}^{s,p}(\Omega).$
	We say that $u$ is a weak solution of \eqref{eq:viscosity1} if only if
	\[
		\int_{\R^{2N}}
				\dfrac{|u(x)-u(y)|^{p-2}(u(x)-u(y))(v(x)-v(y))}{|x-y|^{N+rp}}dxdy
		=\langle f,v\rangle 
	\]
	for every $v\in \widetilde{W}^{s,p}(\Omega)$.
	Here $\langle \cdot,\cdot\rangle$ denotes the duality pairing of $\widetilde{W}^{s,p}(\Omega)$ 
	with $W^{-s,p}(\Omega).$
	
	\medskip
	 
	\noindent{\bf Definition (viscosity solution).}
	Let $p\ge2,$ $f\in C(\overline{\Omega})$ and $u\in C(\R^N)$ be such that $u=0$ in 
	$\Omega^c.$ 
	
	We say that $u$ is a viscosity subsolution
	of \eqref{eq:viscosity1} at a point $x_0\in \Omega$ if and only if
	for any test function $\varphi\in C^2_0(\R^N)$ such that 
			$u(x_0)=\varphi(x_0)$ and $u(x)\le\varphi(x)$ for all $x\in\R^N$ we have that
			\[
				2\int_{\R^N}
				\dfrac{|\varphi(x_0)-\varphi(y)|^{p-2}(\varphi(x_0)-\varphi(y))}{|x_0-y|^{N+sp}} \, dy \le
					f(x_0).
			 \]
	
	We say that $u$ is a viscosity supersolution
				of \eqref{eq:viscosity1} at a point $x_0\in \Omega$ if and only if
					for any test function $\varphi\in C^2_0(\R^N)$ such that 
					$u(x_0)=\varphi(x_0)$ and $u(x)\ge\varphi(x)$ 
					for all $x\in\R^N$ we have that
					\[
						2\int_{\R^N}
						\dfrac{|\varphi(x_0)-\varphi(y)|^{p-2}(\varphi(x_0)-\varphi(y))}{|x_0-y|^{N+sp}} \, dy \ge
							f(x_0).
					 \]

		Finally, $u$ is called a viscosity solution 
			of \eqref{eq:viscosity1} if it is both a 
			viscosity super- and subsolution at $x_0$ for any $x_0\in\Omega$.
 
\medskip	
	
	Following carefully the proof of \cite[Proposition 11]{LL}, we have the following result.
	
	\begin{teo}\label{teo:debilvisco}
		Let $p\ge2$ and $f\in C(\overline{\Omega}).$ If $u$ is a weak solution of  \eqref{eq:viscosity1}
		then it is also a viscosity solution.
	\end{teo}

	The following result is one of the key to show that every eigen-pair associated to the first 
	eigenvalue has constant sign. For the proof we refer to \cite[Lemma 12]{LL}. 
	
	\begin{lem}\label{lema:viscopositivo}
		Let $p\ge 2.$ Assumme $u\ge0$ and $u\equiv0$ in $\Omega^c.$ If $u$ is a 
		viscosity supersolution of $(-\Delta_p)^su=0$ in $\Omega$ then either
		$u>0$ in $\Omega$ or $u\equiv 0$ in $\R^N.$
 	\end{lem}
 	
\section{The eigenvalue problem}\label{section:1erAutov}
	We begin showing that $\lambda_{1,p}$ is the first eigenvalue of
	our problem.
	
	\begin{lem}\label{lema:1A1}
		There is a nontrivial minimizer 
		$(u,v)$ of \eqref{eq:1eraut} such that $u,v > 0$ a.e. in $\Omega$
		and $(u,v)$ is a weak solution of \eqref{eq:autovalores} with
		$\lambda=\lambda_{1,p}.$		
	\end{lem} 
	
	\begin{proof}
		Since $C_0^\infty(\Omega)\times C_0^\infty(\Omega)\subset\wrsp(\Omega),$ we have
		\begin{equation}\label{eq:1A1.1}
			0\le\inf\left\{
					\dfrac{[u]_{r,p}^p+[v]_{s,p}^p}
					{|(u,v)|_{\alpha,\beta}^p}\colon
					(u,v)\in\wrsp(\Omega), uv\not\equiv0
					\right\}<\infty.
		\end{equation}
				 
		Now, we consider a minimizing sequence $\{(u_n,v_n)\}_{n\in\mathbb{N}}$ normalized according to 
		$|(u_n,v_n)|_{(\alpha,\beta)}=1$. It follows from \eqref{eq:1A1.1} that $\{(u_n,v_n)\}$ is bounded in 
		$\wrsp(\Omega).$ Then, by the compactness of the Sobolev embedding theorem, there is a subsequence
		$\{(u_{n_j},v_{n_j})\}_{j\in\mathbb{N}}$ such that 
		\begin{align*}
			&u_{n_j} \rightharpoonup u \text{ weakly in }\widetilde{\mathcal{W}}^{r,p}(\Omega),
			\quad & v_{n_j} \rightharpoonup v \text{ weakly in }\widetilde{\mathcal{W}}^{s,p}(\Omega),\\
			&u_{n_j} \to u \text{ strongly in } L^p(\Omega),
			\quad & v_{n_j} \to v \text{ strongly in } L^p(\Omega).
		\end{align*}
		Thus, $|(u,v)|_{(\alpha,\beta)}=1$ and
		\[
			[u]_{r,p}^p+[v]_{s,p}^p\le\liminf_{j\to\infty}
			\left\{ [u_{n_j}]_{r,p}^p+[v_{n_j}]_{s,p}^p \right\} =\lambda_{1,p}.
		\]
		Therefore $(u,v)$ is a minimizer of \eqref{eq:1eraut}. Moreover, since 
		\[
			[|u|]_{r,p}^p+[|v|]_{s,p}^p\le [u]_{r,p}^p+[v]_{r,p}^p,
		\] 
		we can assume that $u$ and $v$ are non-negative functions.
		
		The fact that this minimizer is a weak solution \eqref{eq:autovalores} with
		$\lambda=\lambda_{1,p}$ is straightforward and can be obtained from the arguments in \cite{LL}. 
		
		Finally, since $u$ and $v$ are non-negative function and $(u,v)$ is a weak solution of
		\eqref{eq:autovalores} with $\lambda=\lambda_{1,p},$ by \cite[Theorem A.1]{Brasco1}, we obtain
		$u,v$ are positive functions a.e. in $\Omega.$
	\end{proof}
	
	The following result follows from the classical inequality
	\[
		||a|-|b||<|a-b|\quad \forall ab<0.
	\]
	
	\begin{co}\label{co:autopositivo}
		If $(u,v)$ is an eigen-pair corresponding to $\lambda_{1,p}$ 
		then $u$ and $v$ have constant sign.
	\end{co}

	Our next aim is to prove that all the eigen-pairs associated to $\lambda_{1,p}$ are bounded.
	For this, we follow ideas from \cite[Theorem 3.2]{Brasco2}.
	
	\begin{lem}\label{lema.cota}
		If $(u,v)$ is an eigen-pair associated to $\lambda_{1,p},$ 
		then $u,v \in L^\infty(\R^N).$		
	\end{lem}		
	
	\begin{proof}
		Without loss of generality we can assume that $r\le s$ and $u,v>0$ a.e. in $\Omega.$ 
		
		It follows from the fractional Sobolev embedding theorem
		(see, e.g., \cite[Corollary 4.53 and Theorem 4.54]{DD}) that,
		if $r>\nicefrac{N}{p}$ then the assertion holds.
		
		Then we need to prove that the assertion also holds in the following 
		cases:
		\begin{description}
			\item[Case 1] $r<\nicefrac{N}{p};$
			\item[Case 2] $r=\nicefrac{N}{p}.$ 
		\end{description} 
		
		Before we start to analyze the different cases, we will show two inequalities.
		For every $M>0,$ we define
		\[
			u_M(x)\coloneqq \min\{u(x),M\}
			\quad\text{ and }\quad v_M(x)\coloneqq \min\{v(x),M\}.
		\]
		Since $(u,v)\in\wrsp(\Omega),$ it is not hard to verify that $(u_M,v_M)\in\wrsp(\Omega).$
		Moreover if $q\ge 1$ then $(u_M^q,v_M^q)\in\wrsp(\Omega).$ Then, 
		since $(u,v)$ is an eigen-pair associated to $\lambda_{1,p},$ 
		$u_M\le u,$ $v_M\le v,$ and $\alpha,\beta\le p,$ we have
		\begin{align*}
			\int_{\R^{2N}}
			\dfrac{|u(x)-u(y)|^{p-2}(u(x)-u(y))(u_M(x)-u_M(y))}{|x-y|^{N+rp}}dxdy
				&\le\lambda_{1,p}\int_{\Omega}u^{\alpha+q-1}v^{\beta} dx,\\
			\int_{\R^{2N}}
			\dfrac{|v(x)-v(y)|^{p-2}(v(x)-v(y))(v_M(x)-v_M(y))}{|x-y|^{N+sp}}dxdy
				&\le\lambda_{1,p}\int_{\Omega}u^{\alpha}v^{\beta+q-1} dx.
		\end{align*}
		Hence, by using \cite[Lemma C2]{Brasco2}, we get  
		\begin{equation}\label{eq:brasco1}
			\begin{aligned}
				\dfrac{qp^p}{q+p-1}
				\int_{\R^{2N}}\frac{|u_M^{\frac{q+p-1}p}(x)-u_M^{\frac{q+p-1}p}(y)|^{p}}{|x-y|^{N+rp}}dxdy
					&\le\lambda_{1,p}\int_{\Omega}u^{\alpha+q-1}v^{\beta} dx,\\
				\dfrac{qp^p}{q+p-1}\int_{\R^{2N}}
				\dfrac{|v_M^{\frac{q+p-1}p}(x)-v_M^{\frac{q+p-1}p}(y)|^{p}}{|x-y|^{N+rp}}dxdy
					&\le\lambda_{1,p}\int_{\Omega}u^{\alpha}v^{\beta+q-1} dx.
			\end{aligned}		
		\end{equation}
		
		We now begin to analyze the different cases.
		
		\noindent{\bf Case 1:} $r<\nicefrac{N}p.$ Since $r\le s,$ then $p_r^\star\le p_s^\star.$
		Therefore, by Sobolev's  embedding theorem,
		\begin{align*}
			\left(\int_{\Omega} u_M^{\frac{q+p-1}{p}p_r^\star} dx\right)^{\frac{p}{p_r^\star}} 
				& \le C(N,r,p,\Omega)\int_{\R^{2N}}\dfrac{|u_M^{\frac{q+p-1}{p}}(x)
				-u_M^{\frac{q+p-1}{p}}(y)|^{p}}{|x-y|^{N+rp}}dxdy,\\
			\left(\int_{\Omega} v_M^{\frac{q+p-1}{p}p_r^\star} dx\right)^{\frac{p}{p_r^\star}} 
				&\le C(N,r,s,p,\Omega)\int_{\R^{2N}}
					\dfrac{|v_M^{\frac{q+p-1}{p}}(x)-v_M^{\frac{q+p-1}{p}}(y)|^{p}}{|x-y|^{N+rp}}dxdy.
		\end{align*}
		Then, by \eqref{eq:brasco1}, we get
		\begin{align*}
			\left(\int_{\Omega} u_M^{\frac{q+p-1}pp_r^\star} dx\right)^{\frac{p}{p_r^\star}} 
				& \le \dfrac{\lambda_{1,p}}{C(N,r,p,\Omega)}\left(\dfrac{q+p-1}{p}\right)^{p-1}
				\int_{\Omega}u^{\alpha+q-1}v^{\beta} dx,\\
			\left(\int_{\Omega} v_M^{\frac{q+p-1}p p_r^\star} dx\right)^{\frac{p}{p_r^\star}} &
				\le \dfrac{\lambda_{1,p}}{C(N,r,s,p,\Omega)}\left(\dfrac{q+p-1}{p}\right)^{p-1}
							\int_{\Omega}u^{\alpha}v^{\beta+q-1} dx.
		\end{align*}
		By using Fatou's lemma and  Young's inequality, we obtain
		\begin{align*}
			\left(\int_{\Omega} u^{\frac{p+p-1}p p_r^\star} dx\right)^{\frac{p}{p_r^\star}} 
				& \le \dfrac{\lambda_{1,p}}{C(N,r,p,\Omega)}
				\left(\dfrac{p+q-1}{p}\right)^{p-1}
					\left(\int_{\Omega}u^{p+q-1} dx+\int_{\Omega} v^{p+q-1} dx\right),\\
			\left(\int_{\Omega} v^{\frac{q+p-1}pp_r^\star} dx\right)^{\frac{p}{p_r^\star}} &
				\le \dfrac{\lambda_{1,p}}{C(N,r,s,p,\Omega)}\left(\dfrac{q+p-1}{p}\right)^{p-1}
									\left(\int_{\Omega}u^{p+q-1} dx+\int_{\Omega} v^{p+q-1} dx\right).
		\end{align*}
		Taking $\mathcal{Q}=\nicefrac{q+p-1}p,$ we get
		\begin{align*}
			\left(\int_{\Omega} u^{\mathcal{Q}\frac{Np}{N-rp}} dx\right)^{\frac{\mathcal{Q}(N-rp)}{\mathcal{Q}N}} 
						& \le \dfrac{\lambda_{1,p}}{C(N,r,p,\Omega)}
						\mathcal{Q}^{p-1}
							\left(\int_{\Omega}u^{\mathcal{Q}p} dx+\int_{\Omega} v^{\mathcal{Q}p} dx\right),\\
			\left(\int_{\Omega} v^{\mathcal{Q}\frac{Np}{N-rp}}dx\right)^{\frac{\mathcal{Q}(N-rp)}{\mathcal{Q}N}} &
						\le \dfrac{\lambda_{1,p}}{C(N,r,s,p,\Omega)}\mathcal{Q}^{p-1}
											\left(\int_{\Omega}u^{\mathcal{Q}p} dx+\int_{\Omega} v^{\mathcal{Q}p} dx\right).
		\end{align*}
		Then
		\begin{align*}
			\|u\|_{L^{\frac{\mathcal{Q}N}{N-rp}p}(\Omega)}^{\mathcal{Q}p}
					 & \le \dfrac{\lambda_{1,p}}{C(N,r,p,\Omega)}
						\mathcal{Q}^{p-1}
						\left(\|u\|_{L^{\mathcal{Q}p}(\Omega)}^{\mathcal{Q}p} +
						\|v\|_{L^{\mathcal{Q}p}(\Omega)}^{\mathcal{Q}p}\right),\\
			\|v\|_{L^{\frac{\mathcal{Q}N}{N-rp}p}(\Omega)}^{\mathcal{Q}p}
								 & \le \dfrac{\lambda_{1,p}}{C(N,r,s,p,\Omega)}
									\mathcal{Q}^{p-1}
									\left(\|u\|_{L^{\mathcal{Q}p}(\Omega)}^{\mathcal{Q}p} +
									\|v\|_{L^{\mathcal{Q}p}(\Omega)}^{\mathcal{Q}p}\right).\\
		\end{align*}
		Hence		
		\begin{align*}
			&\left(\|u\|_{L^{\frac{\mathcal{Q}N}{N-rp}p}(\Omega)}^{\mathcal{Q}p}
				+\|v\|_{L^{\frac{\mathcal{Q}N}{N-rp}p}(\Omega)}^{\mathcal{Q}p}\right)^{\frac1{\mathcal{Q}p}}\\
			& \le \left(\dfrac{2\lambda_{1,p}}{C(N,r,s,p,\Omega)}\right)^{\frac{1}{\mathcal{Q}}}
								\left(\mathcal{Q}^{\frac1{\mathcal{Q}}}\right)^{\frac{p-1}p}
								\left(\|u\|_{L^{\mathcal{Q}p}(\Omega)}^{\mathcal{Q}p} +
								\|v\|_{L^{\mathcal{Q}p}(\Omega)}^{\mathcal{Q}p}\right)^{\frac{1}{\mathcal{Q}p}}.
		\end{align*}
		Now, taking the following sequence
		\[
			\mathcal{Q}_0=1 \quad\text{ and }\quad \mathcal{Q}_{n+1}=\mathcal{Q}_n\dfrac{N}{N-rp}
		\]
		we have
		\begin{align*}
			&\left(\|u\|_{L^{\mathcal{Q}_{n+1}p}(\Omega)}^{\mathcal{Q}_{n}p}
						+\|v\|_{L^{\mathcal{Q}_{n+1}p}(\Omega)}^{\mathcal{Q}_{n}p}
						\right)^{\frac1{\mathcal{Q}_{n}p}}\\
					& \le \left(\dfrac{2\lambda_{1,p}}{C(N,r,s,p,\Omega)}\right)^{\frac{1}{\mathcal{Q}_{n}p}}
							\left(\mathcal{Q}_n^{\frac1{\mathcal{Q}_n}}\right)^{\frac{p-1}p}
							\left(\|u\|_{L^{\mathcal{Q}_np}(\Omega)}^{\mathcal{Q}_np} +
							\|v\|_{L^{\mathcal{Q}_np}(\Omega)}^{\mathcal{Q}_np}
							\right)^{\frac{1}{\mathcal{Q}_np}}
		\end{align*}
		for all $n\in\mathbb{N}.$ Moreover, since  $$\mathcal{Q}_{n+1}=\nicefrac{\mathcal{Q}_nN}{(N-rp)}$$
		we have that
		\begin{align*}
			&\left(\|u\|_{L^{\mathcal{Q}_{n+1}p}(\Omega)}^{\mathcal{Q}_{n}p}
								+\|v\|_{L^{\mathcal{Q}_{n+1}p}(\Omega)}^{\mathcal{Q}_{n}p}
								\right)^{\frac1{\mathcal{Q}_{n}p}}\\
							& \le \left(\dfrac{2\lambda_{1,p}}{C(N,r,s,p,\Omega)}\right)^{\frac{1}{\mathcal{Q}_{n}p}}
												\left(\mathcal{Q}_n^{\frac1{\mathcal{Q}_n}}\right)^{\frac{p-1}p}
												\left(\|u\|_{L^{\mathcal{Q}_np}(\Omega)}^{\mathcal{Q}_{n-1}p} +
												\|v\|_{L^{\mathcal{Q}_np}(\Omega)}^{\mathcal{Q}_{n-1}p}
												\right)^{\frac{1}{\mathcal{Q}_{n-1}p}}
		\end{align*}
		for all $n\ge2.$
		
		Then, iterating the last inequality, we get
		\begin{equation}\label{eq:labestia}
			\begin{aligned}
				&\left(\|u\|_{L^{\mathcal{Q}_{n+1}p}(\Omega)}^{\mathcal{Q}_{n}p}
					+\|v\|_{L^{\mathcal{Q}_{n+1}p}(\Omega)}^{\mathcal{Q}_{n}p}
					\right)^{\frac1{\mathcal{Q}_{n}p}}\\
				&\le \left(\dfrac{2\lambda_{1,p}}{C(N,r,s,p,\Omega)}\right)^{\frac1p\sum_{i=0}^{n}\frac{1}{\mathcal{Q}_i}}
					\left(\prod_{i=0}^{n}\mathcal{Q}_i^{\frac{1}{\mathcal{Q}_i}}\right)^{\frac{p-1}p}
					\left(\|u\|_{L^{p}(\Omega)}^p +\|v\|_{L^{p}(\Omega)}^p \right)^{\frac1p}
			\end{aligned}	
		\end{equation}
		for all $n\ge2.$
		
		Observe that $\mathcal{Q}_n\to\infty$ as $n\to\infty$ due to the fact that $\nicefrac{N}{N-rp}>1.$ Moreover,
		\[
			\sum_{i=0}^{\infty}\frac{1}{\mathcal{Q}_i}=\dfrac{N}{rp}
			\quad\mbox{ and }\quad 
			\prod_{i=0}^{\infty}\mathcal{Q}_i^{\frac{1}{\mathcal{Q}_i}}
			=\left(\frac{N}{N-rp}\right)^{\frac{N}{rpp_r^{\star}}}.
		\]
		Hence, passing to the limit in \eqref{eq:labestia}, we deduce
		\begin{align*}
		&\max\{\|u\|_{L^{\infty}(\Omega)},\|v\|_{L^{\infty}(\Omega)}\}\\
						&\le 
						\left(\dfrac{2\lambda_{1,p}}{C(N,r,s,p,\Omega)}\right)^{\frac{N}{rp^2}}
							\left(\frac{N}{N-rp}\right)^{\frac{N}{rpp_r^{\star}}\frac{p-1}p}
							\left(\|u\|_{L^{p}(\Omega)}^p +\|v\|_{L^{p}(\Omega)}^p \right)^{\frac1p},
		\end{align*}
		that is $u,v\in L^\infty(\Omega).$

		\noindent{\bf Case 2:} $r=\nicefrac{N}p.$ In this case $\wrsp(\Omega)\hookrightarrow 
		L^m(\Omega)\times L^m(\Omega)$ for all $m>1$ then 
		\begin{align*}
			\left(\int_{\Omega} u_M^{\frac{q+p-1}{p}2p} dx\right)^{\frac{1}{2}} 
						& \le C(N,r,p,\Omega)\int_{\R^{2N}}\dfrac{|u_M^{\frac{q+p-1}{p}}(x)
						-u_M^{\frac{q+p-1}{p}}(y)|^{p}}{|x-y|^{N+rp}}dxdy,\\
			\left(\int_{\Omega} v_M^{\frac{q+p-1}{p}2p} dx\right)^{\frac{1}{2}} 
						&\le C(N,r,s,p,\Omega)\int_{\R^{2N}}
							\dfrac{|v_M^{\frac{q+p-1}{p}}(x)-v_M^{\frac{q+p-1}{p}}(y)|^{p}}{|x-y|^{N+rp}}dxdy.
		\end{align*}
		Applying the previous reasoning, we get
		\begin{align*}
			&\left(\|u\|_{L^{2\mathcal{Q}p}(\Omega)}^{\mathcal{Q}p}
						+\|v\|_{L^{2{Q}p}(\Omega)}^{\mathcal{Q}p}\right)^{\frac1{\mathcal{Q}p}}\\
			& \le \left(\dfrac{2\lambda_{1,p}}{C(N,r,s,p,\Omega)}\right)^{\frac{1}{\mathcal{Q}}}
										\left(\mathcal{Q}^{\frac1{\mathcal{Q}}}\right)^{\frac{p-1}p}
										\left(\|u\|_{L^{\mathcal{Q}p}(\Omega)}^{\mathcal{Q}p} +
										\|v\|_{L^{\mathcal{Q}p}(\Omega)}^{\mathcal{Q}p}\right)^{\frac{1}{\mathcal{Q}p}}.
		\end{align*}
		Now, taking the following sequence
		\[
			\mathcal{Q}_0=1 \quad\text{ and }\quad \mathcal{Q}_{n+1}=2\mathcal{Q}_n,
		\]
		the proof follows as in the previous case.			
	\end{proof}
	
		To show that $\lambda_{1,p}$ is simple, we will prove first that $\lambda_{1,p}$ 
	is the unique eigenvalue with the following property: any eigen-pair associated to $\lambda$ has constant sign.	
		
	 \begin{teo}\label{teo:autoval1}
		Let $(u,v)$ be an eigenfunction associated to
	    $\lambda_{1,p}$ such that $u,v\ge0$ in $\Omega.$ If $\lambda>0$ is such that 
	    there is an eigen-pair $(w,z)$  associated to $\lambda$ such that
	    $w,z>0$ then $\lambda=\lambda_1(s,p)$ and
	    there exist $k_1,k_2\in\mathbb{R}$ such that $w = k_1 u$ and $z=k_2v$ a.e. in $\R^N.$
		\end{teo}
		
		\begin{proof}
			Since $\lambda_1(s,p)$ is the first eigenvalue we have
			that $\lambda_1(s,p)\le\lambda$. Moreover, by \cite[Theorem A.1]{Brasco1},
			$u,v>0$ a.e. in $\Omega$ since $(u,v)$ is an eigen-pair associated to $\lambda_{1,p}$
			and $u,v\ge0.$
			 
			Let $k\in\mathbb{N}$ and define $w_k\coloneqq w+\nicefrac1{k},$
			and $z_k\coloneqq z+\nicefrac1{k}.$ We begin proving that $u^{p} / w_k^{p-1}\in
			\widetilde{\mathcal{W}}^{r,p}(\Omega).$
			It is immediate that  $u^{p} / w_k^{p-1}=0$ in $\Omega^c$
			and $w_{k}\in L^{p}(\Omega),$ due to the fact that
			$u\in L^{\infty}(\Omega),$ see
		 	Lemma \ref{lema.cota}.
	
			On the other hand, for any $x,y\in\R^N$
			\begin{align*}
				\Biggl|\frac{u}{w_k}(x)&-\frac{u}{w_k}(y)\Biggr|
				= \Bigg|
				\dfrac{u(x)^{p}-u(y)^p}{w_k(x)^{p-1}}
				+\dfrac{u(y)^p\left(w_k(y)^{p-1}-w_k(x)^{p-1}\right)}
				{w_k(x)^{p-1}w_k(y)^{p-1}}\Bigg|\\
				\le& k^{p-1}\left|u(x)^{p}-u(y)^p\right|
				+\|u\|_{L^{\infty}(\Omega)}^p
				\dfrac{\left|w_k(x)^{p-1}-w_k(y)^{p-1}\right|}
				{w_k(x)^{p-1}w_k(y)^{p-1}}\\
				\le&2\|u\|_{L^{\infty}(\Omega)}^{p-1}k^{p-1}p
				|u(x)-u(y)|\\
				&+\|u\|_{L^{\infty}(\Omega)}^p(p-1)
				\dfrac{w_k(x)^{p-2}+w_k(y)^{p-2}}
				{w_k(x)^{p-1}w_k(y)^{p-1}}|w_k(x)-w_k(y)|\\
				\le& 2\|u\|_{L^{\infty}(\Omega)}^{p-1}k^{p-1}p|u(x)-u(y)|\\
				&+\|u\|_{L^{\infty}(\Omega)}^p(p-1)k^{p-1}
				\left(\dfrac1{w_k(x)}+\dfrac1{w_k(y)}\right)|w(y)-w(x)|\\
		 		\le& C(k,p,\|u\|_{L^{\infty}(\Omega)})
				\left(|u(x)-u(y)|+|w(x)-w(y)|\right).
			\end{align*}
			Hence, we have that
			$u^p/w_{k}^{p-1}\in \widetilde{\mathcal{W}}^{r,p}(\Omega)$
			for all $k\in\mathbb{N}$ since $u,w\in
			\widetilde{\mathcal{W}}^{r,p}(\Omega).$
			Analogously $v^p/z_{k}^{p-1}\in \widetilde{\mathcal{W}}^{s,p}(\Omega).$
			
			Set 
			\[
				L(\varphi,\psi)(x,y)=
		   		|\varphi(x)-\varphi(y)|^p -(\psi(x)-\psi(y))^{p-1}
		   		\left(\dfrac{\varphi(x)^p}{\psi(x)^{p-1}}
		   		-\dfrac{\varphi(y)^p}{\psi(y)^{p-1}}\right)
			\]
			for all functions $\varphi\ge0$ and $\psi>0.$
			By \cite[Lemma 6.2]{Am}, for any $\varphi\ge0$ and $\psi>0$
			\[
				L(\varphi,\psi)(x,y)\ge 0 \quad \forall(x,y) 
			\]
			Then, 
			\begin{align*}
				0&\le \int_{\Omega^2} \dfrac{L(u,w_k)(x,y)}{{|x-y|^{N+rp}}} dxdy +
								\int_{\Omega^2} \dfrac{L(v,z_k)(x,y)}{|x-y|^{N+sp}} dxdy\\
				&\le \int_{\R^{2N}} \dfrac{L(u,w_k)(x,y)}{{|x-y|^{N+rp}}} dxdy +
				\int_{\R^{2N}} \dfrac{L(v,z_k)(x,y)}{|x-y|^{N+sp}} dxdy\\
				&=\lambda_{1,p}\int_{\Omega} |u|^{\alpha}|v|^{\beta} \, dx
				-\lambda\dfrac{\alpha}{p}\int_{\Omega} w^{\alpha-1}z^{\beta} \dfrac{u^{p}}{w_k^{p-1}} dx
				-\lambda\dfrac{\beta}{p}\int_{\Omega} w^{\alpha}z^{\beta-1} \dfrac{v^{p}}{z_k^{p-1}} dx
			\end{align*}
			for all $k\in\mathbb{N},$
			since $(u,v),(w,z)$ are eigen-pairs associated to $\lambda_{1,p}$ and $\lambda,$ respectively.
						
			On the other hand, by Young's inequality,
			\[
				\int_{\Omega} w^{\alpha}z^{\beta}\dfrac{u^{\alpha}v^{\beta}}{w_k^{\alpha}z_k^{\beta}}
				dx \le \dfrac{\alpha}{p}\int_{\Omega} w^{\alpha-1}z^{\beta} \dfrac{u^{p}}{w_k^{p-1}} dx
				+ \dfrac{\beta}{p}\int_{\Omega} w^{\alpha}z^{\beta-1} \dfrac{v^{p}}{z_k^{p-1}} dx
			\]
			for all $k\in\mathbb{N}.$ Then
			\begin{align*}
				0&\le \int_{\Omega} \dfrac{L(u,w_k)(x,y)}{{|x-y|^{N+rp}}} dxdy +
				\int_{\Omega} \dfrac{L(v,z_k)(x,y)}{|x-y|^{N+sp}} dxdy\\
				&\le\lambda_{1,p}\int_{\Omega} |u|^{\alpha}|v|^{\beta} \, dx
				-\lambda \int_{\Omega} w^{\alpha}z^{\beta}\dfrac{u^{\alpha}v^{\beta}}{w_k^{\alpha}z_k^{\beta}}dx.
			\end{align*}
			By Fatou's lemma and the dominated convergence theorem we obtain
			\[	
				0\le \int_{\Omega^2} \dfrac{L(u,w)(x,y)}{{|x-y|^{N+rp}}} dxdy +
							\int_{\Omega^2} \dfrac{L(v,z)(x,y)}{|x-y|^{N+sp}} dxdy
							\le (\lambda_{1,p}-\lambda)\int_{\Omega} |u|^{\alpha}|v|^{\beta} \, dx.
			\]
			Then $\lambda=\lambda_{1,p}$ and $L(u,w)=0$ and $L(v,z)=0$ a.e. in $\Omega.$
			
			Finally, again by \cite[Lemma 6.2]{Am}, there exist $k_1,k_2\in\mathbb{R}$ such that $w = k_1 u$ 
			and $z=k_2v$ a.e. in $\R^N.$
		\end{proof}
	
	Now, we show that $\lambda_{1,p}$ is simple.
	
	\begin{co}
		Let $(u_1,v_1)$ be an eigen-pair associated to 
		$\lambda_{1,p}$ normalized according to $|(u_1,v_1)|_{\alpha,\beta}=1.$
		If 	$(u,v)$ is an eigen-pair associated to 
		$\lambda_{1,p}$ then there is a constant $k$ such that
		$(u,v)=k(u_1,v_1).$
	\end{co}
	\begin{proof}
		By Theorem \ref{teo:autoval1}, there exist $k_1$ and $k_2$ such that 
		$u=k_1u_1$ and $v=k_2v_2.$  Without loss of generality, we can assume that $k_1\le k_2.$
		
		Then, since $(u_1,v_1)$ and $(u,v)$ are eigen-pairs 
		associated to the first eigenvalue $\lambda_{1,p}$ and $|(u,v)|_{\alpha,\beta}=1,$ we get
		\[
			\left(\left(\dfrac{k_1}{k_2}\right)^{\beta}-1\right)
			[u]^p_{r,p}+
			\left(\left(\dfrac{k_2}{k_1}\right)^{\alpha}-1\right)
			[v]^p_{s,p}=0.
		\]
		Taking $x=k_1/k_2,$ $a=[u]^p_{r,p}$ and $b=[v]^p_{s,p},$ we get
		\[
			a(x^{\beta}-1)+b\dfrac{1-x^\alpha}{x^{\alpha}}=0.
		\]
		Multiplying by $x^{\alpha}$ and by using that $\alpha+\beta=p,$ we obtain 
		\[	
			ax^{p}-(a+b)x^{\alpha}+b=0.
		\]
		
		To end the proof, we only need to show that 1 is the unique zero of the function
		\[
			f\colon[0,1]\to\R, \quad f(x)=ax^{p}-(a+b)x^{\alpha}+b.
		\]
		Observe that, for any $x\in(0,1)$ we have 
		\[
			f^\prime(x)=pa x^{\alpha-1}\left(x^{p-\alpha}-\frac{a+b}{a}\dfrac{\alpha}{p}\right)
			=pax^{\alpha-1}\left(x^{\alpha}-\frac{a+b}{a}\dfrac{\alpha}{p}\right).
		\]
		
		On the other hand, since $(u_1,v_1)$ is an eigen-pair 
		associated to $\lambda_{1,p}$ such that $|(u,v)|_{\alpha,\beta}=1,$ we have
		\[
			a+b=\lambda_{1,p} \quad\text{ and }\quad a=\dfrac{\alpha}p\lambda_{1,p},
		\]
		then
		\[
			\dfrac{a+b}{a}=\dfrac{p}{\alpha},
		\]
		that is
		\[
			\dfrac{a+b}{a}\dfrac{\alpha}{p}=1.
		\]
		Hence
		\[
			f^\prime(x)<0 \quad\forall x\in(0,1).
		\]
		that is $f$ is decreasing. Therefore $x=1$ is the unique zero of $f.$ 
	\end{proof}
	
	Recall that we made the assumption:
	\[
			\min\{\alpha,\beta\}\ge1.
	\]
	Now, if $(u,v)$ is an eigen-pair associated to $\lambda_{1,p}$  then
	$$|u|^{\alpha-2}u|v|^\beta,|u|^{\alpha}|v|^{\beta-2}v\in L^\infty(\Omega)$$ 
	due to Lemma \ref{lema.cota}.
	Thus, by \cite[Theorem 1.1]{regularidad}, we have the following result.
	
	\begin{lem} \label{lema:regularida}
		If $(u,v)$ is an eigen-pair associated to $\lambda_{1,p},$ 
	 	then there exist $\gamma_1=\gamma_1(N,p,r)\in(0,r]$ and
	 	$\gamma_2=\gamma_2(N,p,s)\in(0,s]$ such that
	 	$(u,v)\in C^{\gamma_1}(\overline{\Omega})\times
	 	C^{\gamma_2}(\overline{\Omega}).$ 
	\end{lem}

	Thus, by Lemma \ref{lema:regularida} and Theorem \ref{teo:debilvisco},
	we have that	
	
	\begin{co} \label{cor:autovisco} 
		If  $(u,v)$ is an eigen-pair associated to $\lambda_{1,p}$
		 then $u$ is a viscosity solution of 
		 \[
		 	\begin{cases}
		 		(-\Delta_p)^r u= \lambda_{1,p}\dfrac{\alpha}{p} |u|^{\alpha-2}u|v|^{\beta} &\text{ in } \Omega,\\
		 		u=0&\text{ in } \R^N\setminus\Omega,
		 	\end{cases}
		 \]
		 and $v$ is a viscosity solution of 
		 \[
		 	\begin{cases}
		 		 (-\Delta_p)^s v= \lambda_{1,p}\dfrac{\beta}{p} |u|^{\alpha}|v|^{\beta-2}v &\text{ in } \Omega,\\
		 		 v=0&\text{ in } \R^N\setminus\Omega,
		 	\end{cases}
		 \]
	\end{co}
	
	Therefore, by Corollary \ref{cor:autovisco} and Lemma \ref{lema:viscopositivo}, we get
	\begin{co}\label{co:autopositivo*}
		If $(u,v)$ is an eigen-pair corresponding to the first eigenvalue $\lambda_{1,p}$, 
		then  $|u|,|v|>0$ in $\Omega$.
	\end{co}

	Finally, we show that there is a sequence of eigenvalues.
	
	\begin{lem}\label{lem:sucesion}
		There is a sequence of eigenvalues $\lambda_n$ such that
			$\lambda_n\to\infty$ as $n\to\infty$.
	\end{lem}
	\begin{proof}			
		We follow ideas from \cite{GAP} and hence we omit the details.
		Let us consider 
		\[
			M_\tau = \{(u,v) \in \wrsp(\Omega)\colon  
			[u]_{r,p}^p+[v]_{s,p}^p= p \tau \}
		\] 
		and
		\[ 
			\varphi (u,v) = \frac{1}{p}
				\int_{\Omega} |u|^\alpha|v|^\beta dx.
		\]
		We are looking for critical points
		of $\varphi$ restricted to the manifold $M_\tau$ using a minimax technique.
		We consider the class
		\[
	  		\Sigma = \{A\subset \wrsp(\Omega)\setminus\{0\}
				\colon A \mbox{ is closed, } A=-A\}.
	  	\]
		Over this class we define the genus, 
		$\gamma\colon\Sigma\to {\mathbb{N}}\cup\{\infty\}$, as 
		\[
			\gamma(A) = \min\{k\in {\mathbb{N}}\colon
			\mbox{there exists } \phi\in C(A,{{\mathbb{R}}}^k-\{0\}),
			\ \phi(x)=-\phi(-x)\}.
		\]
		Now, we let $C_k = \{ C \subset M_\tau \colon C 
		\mbox{ is compact, symmetric and } \gamma ( C) \le k \} $ 
		and let
		\begin{equation}
				\label{betak} \beta_k = \sup_{C \in C_k} \min_{(u,v) \in C} \varphi(u,v). 
		\end{equation}
		Then $\beta_k >0$ and there is $(u_k,v_k) \in
		M_\tau$ such that $\varphi (u_k,v_k) = \beta_k$ and $(u_k,v_k)$ is a weak
		eigen-pair with $\lambda_k = \nicefrac{\tau}{\beta_k}.$
	\end{proof}


\section{The limit as $p\to \infty$} \label{sec-p-infty}
	From now on, we assume that \eqref{eq:alfabeta} and \eqref{lim.Gamma} hold.
	Recall that we defined $\Lambda_{1,\infty} $ by
	$$
			\Lambda_{1,\infty} 
			=  \inf \left\{
			\frac{\max \{ [u]_{r,\infty} ; [v]_{s,\infty} \} }{ \| |u|^{\Gamma} |v|^{1-\Gamma} \|_{L^\infty (\Omega)} }
			\colon (u,v)\in \wrsi(\Omega)\right\}.
		$$
	First, we show the geometric characterization of
	$\Lambda_{1,\infty}.$  Then, we prove that there exists a sequence of 
	eigen-pairs $(u_p,v_p)$ associated to $\lambda_{1,p}$ such that 
	$(u_p,v_p)\to(u_\infty,v_\infty)$ as $p\to \infty$ and $(u_\infty,v_\infty)$ is a minimizer for
	$\Lambda_{1,\infty}.$ Finally we will show that 
	$(u_\infty,v_\infty)$  is a viscosity solution of \eqref{eq:limite}.    		
	\medskip	
	
	\subsection{Geometric characterization} Observe that, by Arzel\`a--Ascoli theorem, there exists a minimizer
	for $\Lambda_{1,\infty}.$ Moreover, if $(u,v)$ is a minimizer for  $\Lambda_{1,\infty}$ then
	so is $(|u|,|v|).$	Now, we show the geometric characterization of
	$\Lambda_{1,\infty}.$   		

	\begin{lem} \label{lema:caracgeom}
		The following equality holds 
		$$
			\Lambda_{1,\infty} = \left[ \frac{1}{R(\Omega)} \right]^{ (1-\Gamma) s + \Gamma r }.
		$$
	\end{lem}
	\begin{proof}
		Let us take $(u,v)$ a minimizer for $\Lambda_{1,\infty}$
		with $u,v\ge0$ normalized according to $\| u^{\Gamma} v^{1-\Gamma} \|_{L^\infty (\Omega)}=1.$
		Therefore, there is a point $x_0 \in \Omega$ such that
		$$
			u^{\Gamma} (x_0) v^{1-\Gamma} (x_0) =1.
		$$
		Let us call 
		$$
			a= u (x_0)\qquad \mbox{and} \qquad b= v (x_0).
		$$
		Then, since $u,v=0$ in $\Omega^c$,
		$$
			[u]_{r,\infty} = \sup_{(x,y)\in \overline{\Omega}} \frac{| u(y) - u(x)|}{|x-y|^{r}} \geq 
			\frac{a}{[\dist (x_0, \partial 
			\Omega)]^r}
		$$
		and
		$$
			[v]_{s,\infty} = \sup_{(x,y)\in\overline{\Omega}} 
			\frac{| v(y) - v(x)|}{|x-y|^{s}} \geq \frac{b}{[\dist (x_0, \partial \Omega)]^s}.
		$$
		Therefore, we are left with
		$$
			\Lambda_{1,\infty}  \geq 
			\inf_{ (a,b,x_0)\in \mathcal{A}} \left\{
			\max\left\{ \frac{a}{[\dist (x_0, \partial \Omega)]^r} ; \frac{b}{[\dist (x_0, \partial \Omega)]^s}
			\right\} 
			\right\},  
		$$
		where
		\[
			\mathcal{A}\coloneqq\{(0,\infty)\times(0,\infty)\times\overline{\Omega}
						\colon a^{\Gamma}  b^{1-\Gamma} =1\} .
		\]
		
		To compute the infimum we observe that we must have
		$$
			\frac{a}{[\dist (x_0, \partial \Omega)]^r} 
			= \frac{b}{[\dist (x_0, \partial \Omega)]^s}
		$$
		that is,
		$$
			a= b [\dist (x_0, \partial \Omega)]^{r-s}.
		$$
		Then, using $a^{\Gamma}  b^{1-\Gamma} =1$, we obtain
		$$
			b [\dist (x_0, \partial \Omega)]^{\Gamma (r-s)} =1.
		$$
		Hence
		$$
			b = [\dist (x_0, \partial \Omega)]^{\Gamma (s-r)} 
		$$
		and
		$$
			a= [\dist (x_0, \partial \Omega)]^{(r-s) (1-\Gamma)}.
		$$
		Therefore, we are left with
		$$
			\inf_{x_0} [\dist (x_0, \partial \Omega)]^{- [ (1-\Gamma) s + \Gamma r ]},
		$$
		that is attained at a point $x_0\in \Omega$ 
		that maximizes the distance to the boundary. That is, letting
		$$
			R(\Omega) = \dist (x_0, \partial \Omega),
		$$
		we obtain that
		$$
			\Lambda_{1,\infty} \geq \left[ \frac{1}{R(\Omega)} \right]^{ (1-\Gamma) s + \Gamma r }.
		$$
		
		To end the proof, we need to show the reverse inequality. As before, let $x_0\in \Omega$
		be the point where is attained the maximum distance to the boundary.	Set
		\begin{align*}
			u_0(x)&=R(\Omega)^{(r-s)(1-\Gamma)}\left(1-\dfrac{|x-x_0|}{R(\Omega)}\right)_+^{r},\\
			v_0(x)&=R(\Omega)^{-(r-s)\Gamma}\left(1-\dfrac{|x-x_0|}{R(\Omega)}\right)_+^{s}.
		\end{align*}
		We can observe that $(u_0,v_0)\in C^r(\R^N)\times C^s(\R^N),$ $\|u^{\Gamma}_0
		v^{1-\Gamma}_0 \|_{L^{\infty}(\Omega)}=1$
		and
		\[
			\max\{[u_0 ]_{r,\infty};[v_0 ]_{s,\infty}\}\le \left[ \frac{1}{R(\Omega)} \right]^{ (1-\Gamma) s + \Gamma r }.
		\]
		Therefore
		$$
			\Lambda_{1,\infty} =  \inf \left\{
			\frac{\max \{ [u]_{r,\infty} ; [v]_{s,\infty} \} }{ \| |u|^{\Gamma} |v|^{1-\Gamma} \|_{L^\infty (\Omega)} }
			\colon (u,v)\in \wrsi(\Omega)\right\} \leq 
			\left[ \frac{1}{R(\Omega)} \right]^{ (1-\Gamma) s + \Gamma r }.
		$$
	\end{proof}
	
	\begin{re} Observe that $(u_0,v_0)$ is a minimizer of $\Lambda_{1,\infty}.$
	\end{re}
	
	\subsection{Convergence} Now, we prove that there exists a sequence of
	eigen-pairs $(u_p,v_p)$ associated to $\lambda_{1,p}$ such that 
	$(u_p,v_p)\to(u,v)$ as $p\to \infty$ and $(u,v)$ is a minimizer for
	$\Lambda_{1,\infty}.$
	\begin{lem} \label{lema.conv.unif.autov} 
		Let $(u_{p},v_{p})$ be an 
		eigen-pair for $\lambda_{1,p}$ such that $u_p$ and $v_p$ are positive and 
		$|(u,v)|_{\alpha,\beta}=1$. Then, there exists a sequence 
		$p_j \to \infty$ such that
		\[
			(u_{p_j},v_{p_j}) \to (u_\infty,v_\infty)
		\]
		uniformly in ${\mathbb{R}}^N$.  The limit $(u_\infty,v_\infty)$ belongs to the space 
		$\wrsi(\Omega)$ and is a minimizer of $\Lambda_{1,\infty}.$
		In addition, it holds that
		$$
			[\lambda_{1,p}]^{\nicefrac{1}{p}} \to \Lambda_{1,\infty}.
		$$
	\end{lem}
	
	\begin{proof}
		We start showing that
		\begin{equation}\label{eq:limt1}
			\limsup_{p\to\infty}[\lambda_{1,p}]^{
			\nicefrac1p}\le \Lambda_{1,\infty}.
		\end{equation}		
		Let $\gamma>1$ be such that $\gamma\max\{r,s\}<1.$ 
		Then $(u_\gamma,v_\gamma)=(u_\infty^\gamma,v_\infty^\gamma)\in\wrsp(\Omega)\cap\wrsi (\Omega)$ for all $p>1.$
		Thus
		\[
			[\lambda_{1,p}]^{\nicefrac{1}{p}}\le \dfrac{\left([u_\gamma]_{r,p}^p+[v_\gamma]_{s,p}^p\right)^{\nicefrac{1}{p}}}{|(u_\gamma,
			v_\gamma)|_{\alpha,\beta}}
		\]  
		for all $p>1.$ In addition, we observe that $\|u^{\Gamma}_\gamma v^{1-\Gamma}_\gamma \|_{L^{\infty}(\Omega)}=1.$ Then
		\begin{align*}
			\limsup_{p\to\infty}
			[\lambda_{1,p}]^{\nicefrac{1}{p}}&\le 
			\max\left\{[u_\gamma]_{r,\infty};[v_\gamma]_{s,\infty}\right\}\\
			&\le \max\left\{2^{r(\gamma-1)}R(\Omega)^{\gamma(r-s)(1-\Gamma)-r};
				2^{s(\gamma-1)}R(\Omega)^{-\gamma(r-s)\Gamma-s}
			\right\}.
		\end{align*}				 
		Therefore, passing to the limit as $\gamma\to 1$ in the previous inequality and using Lemma \ref{lema:caracgeom},	
		we get \eqref{eq:limt1}.

		Our next step is to show that
		\begin{equation}\label{eq:limt2}
			\Lambda_{1,\infty} \le
			\liminf_{p\to\infty}[\lambda_{1,p}]^{\nicefrac1{p}}.
		\end{equation}
		Let $p_j>1$ be such that
		\begin{equation}\label{eq:inf}
			\liminf_{p\to\infty}[\lambda_{1,p}]^{\nicefrac1{p}}
			=\lim_{j\to\infty}[\lambda_{j}]^{\nicefrac1{p_j}},
		\end{equation}
		where $\lambda_j=\lambda_{1,p_j}.$ By \eqref{eq:limt1}, 
		without of loss of generality, we can assume  
		$$
			2\max\{\nicefrac{N}r,\nicefrac{N}s\}< p_1,\quad p_j\le p_{j+1},\quad  \text{and }
		$$
		\begin{equation}\label{eq:vie}
			[\lambda_{j}]^{\nicefrac1{p_j}}
			=\left([u_{j}]_{r,p_j}^{p_j}+[v_{j}]_{s,p_j}^{p_j}\right)^{\nicefrac{1}{p_j}}
			\le \Lambda_{1,\infty} + \varepsilon \qquad \forall j\in\mathbb{N},
		\end{equation}
		where $\varepsilon$ is any positive number and $(u_j,v_j)$ is an eigen-pair corresponding to $\lambda_{j}$ 
		normalized according to $|(u_{j},v_j)|_{\alpha_j,\beta_j}=1$ ($\alpha_j=\alpha_{p_j},$ 
		$\beta_j=\beta_{p_j}$)  and such that $u_j,v_j>0$ in $\Omega.$  
	
		Let $q\in(2\max\{\nicefrac{N}r,\nicefrac{N}s\}, p_1),$ $t_1=r-\nicefrac{N}{q}$ and $t_2=s-\nicefrac{N}{q}.$ 
		It follows from \eqref{eq:vie} and Lemmas \ref{lem:poincare} and \ref{lem:inclusion} that 
		$\{u_j\}$ and $\{v_j\}$ are bounded in $W^{t_1,q}(\Omega)$ and $W^{t_2,q}(\Omega),$ respectively.
		Since $q\min\{t_1,t_2\}\ge N,$ taking a subsequence if is necessary, we get
		\begin{align*}
			u_j\to u_\infty & \text{ strongly in } C^{0,\gamma_1}(\overline{\Omega}),\\
			v_j\to v_\infty & \text{ strongly in } C^{0,\gamma_2}(\overline{\Omega}).
		\end{align*}				
		due to the compact Sobolev embedding theorem. Here $0<\gamma_1<t_1-\nicefrac{N}q=r-2\nicefrac{N}{q}$
		and $0<\gamma_1<t_2-\nicefrac{N}q=s-2\nicefrac{N}{q}.$ Therefore $u_\infty=v_\infty=0$ on $\partial\Omega.$		
		
		On the other hand, by Lemma \ref{lem:inclusion},
		\begin{align*}
				|u_j|_{t_1,q}&\le
				{\rm{diam}}(\Omega)^{\nicefrac{N}{p_j}}|\Omega|^{\nicefrac2q-\nicefrac2{p_j}}|u_j|_{r,p_j}
				\le {\rm{diam}}(\Omega)^{\nicefrac{N}{p_j}}|\Omega|^{\nicefrac2q-\nicefrac2{p_j}}
				[\lambda_{j}]^{\nicefrac1{p_j}},\\
				|v_j|_{t_2,q}&\le
				{\rm{diam}}(\Omega)^{\nicefrac{N}{p_j}}|\Omega|^{\nicefrac2q-\nicefrac2{p_j}}|v_j|_{s,p_j}
				\le  {\rm{diam}}(\Omega)^{\nicefrac{N}{p_j}}|\Omega|^{\nicefrac2q-\nicefrac2{p_j}}
				[\lambda_{j}]^{\nicefrac1{p_j}}.
		\end{align*}				
		Then passing to the limit as $j\to\infty$ and using Fatou's lemma, we get 
		$(u_\infty,v_\infty)\in W^{t_1,q}(\Omega)\times W^{t_2,q}(\Omega)$
		and
		\begin{align*}
				|u_\infty|_{t_1,q}&\le |\Omega|^{\nicefrac2q} \liminf_{p\to\infty}[\lambda_{1,p}]^{\nicefrac1p},\\
				|v_\infty|_{t_2,q}&\le |\Omega|^{\nicefrac2q}\liminf_{p\to\infty}[\lambda_{1,p}]^{\nicefrac1p}.
		\end{align*}				
		Now passing to the limit as $q\to\infty$ we obtain
		\begin{align*}
			[u_\infty]_{r,\infty}\le \liminf_{p\to\infty}[\lambda_{1,p}]^{\nicefrac1p},\\
			[v_\infty]_{s,\infty}\le \liminf_{p\to\infty}[\lambda_{1,p}]^{\nicefrac1p},
		\end{align*}
		that is $(u_\infty,v_\infty)\in\wrsi(\Omega)$ and 
		\begin{equation}\label{eq:yaestamos}
			\max\{[u_\infty]_{r,\infty};[v_\infty]_{r,\infty}\}\le \liminf_{p\to\infty}[\lambda_{1,p}]^{\nicefrac1p}.		
		\end{equation}
	
		To end the proof we only need to show that $\|u_\infty^{\Gamma}v_\infty^{1-\Gamma}\|_{L^{\infty}(\Omega)}=1.$
		For all $q>1$ there exists $j_0\in\mathbb{N}$ such that
		$p_j>q$ if $j>j_0$ and therefore, by Fatou's Lemma and H\"older's inequality, 
		 we get 
		\[
			\|u^{\Gamma}_\infty v^{1-\Gamma}_\infty\|_{L^q(\Omega)}^q\le
			\liminf_{j\to\infty}\int_{\Omega} u_j^{{\nicefrac{\alpha_j}{p_j}}q} v_j^{{\nicefrac{\beta_j}{p_j}}q}dx
			\le \liminf_{j\to\infty} |\Omega|^{1-\frac{q}{p_j}}=1
		\]
		due to $|(u_j,v_j)|_{\alpha_j,\beta_j}=1.$
		Then passing to the limit as $q\to\infty$ we have 
		\[
			\|u_\infty^{\Gamma}v_\infty^{1-\Gamma}\|_{L^{\infty}(\Omega)}\le 1.
		\]
		
		On the other hand
		\[
		 	1=|(u_j,v_j)|_{\alpha_j,\beta_j}^{\nicefrac{1}{p_j}}\le 
		 	\|u_j^{\nicefrac{\alpha_j}{p_j}}v_j^{\nicefrac{\beta_j}{p_j}}\|_{L^{\infty}(\Omega)}|\Omega|^{\nicefrac1{p_j}}\to
		 	\|u_\infty^{\Gamma}v_\infty^{1-\Gamma}\|_{L^{\infty}(\Omega)}.
		\]
		Therefore $\|u_\infty^{\Gamma}v_\infty^{1-\Gamma}\|_{L^{\infty}(\Omega)}=1.$
	\end{proof}
	
	\subsection{Viscosity Solution}
		Finally we will show that $(u_\infty,v_\infty)$ is a viscosity solution of 
		\begin{equation}\label{eq:limite}
			\begin{cases}
				\min\left\{\mathcal{L}_{r,\infty}u(x);\mathcal{L}_{r,\infty}^+u(x)-\Lambda_{1,\infty} u^{\Gamma}(x) v^{1-\Gamma}(x)\right\}
				=0
				&\text{ in } \Omega,\\
				\min\left\{\mathcal{L}_{s,\infty}u(x);\mathcal{L}_{s,\infty}^+u(x)-\Lambda_{1,\infty} u^{\Gamma}(x) v^{1-\Gamma}(x)\right\}
				=0&\text{ in } \Omega,\\
				u=v=0 &\text{ in } \R^N\setminus\Omega,
			\end{cases}
		\end{equation}
		where 
		\[
			\mathcal{L}_{t,\infty}w(x)=\mathcal{L}_{t,\infty}^+w(x)
			+\mathcal{L}_{r,\infty}^-w(x)= \sup_{y\in\R^N}\dfrac{w(x)-w(y)}{|x-y|^{t}}
			+\inf_{y\in\R^N}\dfrac{w(x)-w(y)}{|x-y|^{t}}.
		\]
				
		Let us introduce the precise definition of viscosity solution of \eqref{eq:limite}.

		\medskip
		
		\noindent{\bf Definition.} Let $(u,v) \in C(\R^N)\times C(\R^N)$ be such that $u,v\ge0$ in $\Omega$
		and $u=v=0$ in $\Omega^c.$ 
		
		 We say that $(u,v)$ is a viscosity subsolution
				of \eqref{eq:limite} at a point $x_0\in \Omega$ if and only if
				for any test pair $(\varphi,\psi)\in C^2_0(\R^N)\times C^2_0(\R^N)$ 
				such that 	$u(x_0)=\varphi(x_0),$ $v(x_0)=\psi(x_0),$ 
				$u(x)\le\varphi(x)$	 and $v(x)\le\psi(x)$ for all $x\in\R^N$ we have that
				\begin{align*}
					&\min\{\mathcal{L}_{r,\infty}\varphi(x_0); \mathcal{L}_{r,\infty}^+\varphi(x_0)
					-\Lambda_{1,\infty}u^{\Gamma}(x_0)v^{1-\Gamma}(x_0)\} \le0,\\
					&\min\{\mathcal{L}_{r,\infty}\psi(x_0); \mathcal{L}_{r,\infty}^+\psi(x_0)
					-\Lambda_{1,\infty}u^{\Gamma}(x_0)v^{1-\Gamma}(x_0)\} \le0
				\end{align*}
			
			We say that $(u,v)$ is a viscosity subsolution
				of \eqref{eq:limite} at a point $x_0\in \Omega$ if and only if
				for any test pair $(\varphi,\psi)\in C^2_0(\R^N)\times C^2_0(\R^N)$ 
				such that 	$u(x_0)=\varphi(x_0),$ $v(x_0)=\psi(x_0),$ 
				$u(x)\ge\varphi(x)$	 and $v(x)\ge\psi(x)$ for all $x\in\R^N$ we have that
				\begin{align*}
					&\min\{\mathcal{L}_{r,\infty}\varphi(x_0); \mathcal{L}_{r,\infty}^+\varphi(x_0)
					-\Lambda_{1,\infty}u^{\Gamma}(x_0)v^{1-\Gamma}(x_0)\} \ge0,\\
					&\min\{\mathcal{L}_{r,\infty}\psi(x_0); \mathcal{L}_{r,\infty}^+\psi(x_0)
					-\Lambda_{1,\infty}u^{\Gamma}(x_0)v^{1-\Gamma}(x_0)\} \ge0
				\end{align*}
			
			Finally, $u$ is a viscosity solution
				of \eqref{eq:limite} at a point $x_0\in \Omega$ viscosity solution, if it is both a viscosity super- and subsolution at every $x_0$.
				
				\medskip

		\begin{lem}
		$(u_\infty,v_\infty)$ is a viscosity solution of \eqref{eq:limite}.		
	\end{lem}
	
	\begin{proof}
	It follows as in \cite[Section 8]{LL}, we include a sketch here for completeness. 
	Let us show that $u_\infty$ is a viscosity supersolution of the first
	equation in \eqref{eq:limite} (the fact that it is a viscosity sub solution is similar). 
	Assume that $\varphi$ is a test function touching $u_\infty$ strictly from below at a point $x_0 \in \Omega$. 
	We have that $u_j -\varphi$ has a minimum at points $x_j \to x_0$. Since $u_j$ is a weak solution (and hence a viscosity solution) to the first equation in our system we have the inequality 
	$$
	- (-\Delta_{p_j})^r \varphi (x_j) + \lambda_{1,p_j}\dfrac{\alpha_j}{p_j} |\varphi|^{\alpha_j-2}\varphi |v|^{\beta_j} (x_j) \leq 0.
	$$
	Writing (as in \cite{LL}) 
$$
A_j^{p_j-1} = 2\int_{\R^N} \dfrac{|\varphi(x_j)-\varphi(y)|^{p_j-2}(\varphi(x_j)-\varphi(y))^+}{|x_j-y|^{N+sp_j}} \, dy, $$
$$
B_j^{p_j-1} =  2\int_{\R^N}\dfrac{|\varphi(x_j)-\varphi(y)|^{p_j-2}(\varphi(x_j)-\varphi(y))^-}{|x_j-y|^{N+sp_j}} \, dy $$
and
$$
C_j^{p_j-1} =  \lambda_{1,p_j}\dfrac{\alpha_j}{p_j} |\varphi|^{\alpha_j-2}\varphi |v|^{\beta_j} (x_j)
$$
we get 
$$
A_j^{p_j-1}  + C_j^{p_j-1}  \leq  B_j^{p_j-1} . 
$$	
Using that
$$
A_j \to \mathcal{L}_{r,\infty}^+\varphi(x_0),
 \qquad B_j \to - \mathcal{L}_{r,\infty}^-\varphi(x_0)
\qquad \mbox{and} \qquad C_j \to \Lambda_{1,\infty}u^{\Gamma}(x_0)v^{1-\Gamma}(x_0)
$$
we obtain
$$
\min\{\mathcal{L}_{r,\infty}\varphi(x_0); \mathcal{L}_{r,\infty}^+\varphi(x_0)
					-\Lambda_{1,\infty}u^{\Gamma}(x_0)v^{1-\Gamma}(x_0)\} \leq 0.
$$	
	\end{proof}
	

\bibliographystyle{amsplain}

\end{document}